\documentclass[12pt,reqno]{amsart}

\usepackage[ansinew]{inputenc}
\usepackage{graphicx}
\usepackage[bookmarksnumbered,colorlinks, linkcolor=blue, citecolor=red, pagebackref, bookmarks, breaklinks]{hyperref}

\newtheorem{thm}{Theorem}[section]
\newtheorem{cor}[thm]{Corollary}
\newtheorem{lem}[thm]{Lemma}
\newtheorem{prop}[thm]{Proposition}

\theoremstyle{remark}
\newtheorem{rem}[thm]{Remark}

\numberwithin{equation}{section}

\def\diver{\mathop{\text{\normalfont div}}}

\newcommand{\R}{\mathbb{R}}
\newcommand{\N}{\mathbb{N}}
\newcommand{\J}{\mathcal{J}}
\newcommand{\I}{\mathcal{I}}
\newcommand{\HH}{\mathcal{H}}
\newcommand{\Ks}{\mathcal{K}}

\newcommand{\A}{\mathcal{A}}
\newcommand{\ve}{\varepsilon}
\newcommand{\lam}{\lambda}
\newcommand{\LL}{\mathcal{L}}

\newcommand{\average}{{\mathchoice {\kern1ex\vcenter{\hrule height.4pt
width 6pt depth0pt} \kern-9.7pt} {\kern1ex\vcenter{\hrule
height.4pt width 4.3pt depth0pt} \kern-7pt} {} {} }}

\def\R{\mathbb{R}}

\begin{document}

\title[Homogeneous eigenvalue problems in Orlicz-Sobolev spaces]{Homogeneous eigenvalue problems in Orlicz-Sobolev spaces}

\author{Juli\'an Fern\'andez Bonder, Ariel Salort and Hern\'an Vivas}

\address[JFB and AS]{Instituto de C\'alculo (IC), CONICET\\
Departamento de Matem\'atica, FCEN - Universidad de Buenos Aires\\
Ciudad Universitaria, Pabell\'on I, C1428EGA, Av. Cantilo s/n\\
Buenos Aires, Argentina}

\email[JFB]{jfbonder@dm.uba.ar}

\email[AS]{asalort@dm.uba.ar}

\address[HV]{Instituto de C\'alculo (IC), CONICET\\
Centro Marplatense de Investigaciones matem\'aticas/CONICET, Dean Funes 3350, 7600 Mar del Plata, Argentina}

\email{havivas@mdp.edu.ar}



\begin{abstract}
In this article we consider a homogeneous eigenvalue problem ruled by the fractional $g-$Laplacian operator whose Euler-Lagrange equation is obtained by minimization of a quotient involving Luxemburg norms. We prove existence of an infinite sequence of variational eigenvalues and study its behavior as the fractional parameter $s\uparrow 1$ among other stability results.
\end{abstract}

\maketitle

\section{Introduction and main results}\label{Intro}

Eigenvalue problems are among the most widely studied problems in Partial Differential Equations; this stems in part for their natural appearance in the description of numerous natural phenomena (from vibrating membranes to quantum physics, passing through signal processing and many others) and in part from their intrinsic mathematical interest. In that sense, a rather natural question that arises when dealing with a class of operators is that of existence of eigenvalues. 

The aim of this paper is precisely to (affirmatively) answer the question of existence for a class of homogeneous eigenvalue problems posed in fractional Orlicz-Sobolev spaces, more precisely to prove the existence of a sequence of variational eigenvalues for such operators. Furthermore, we study asymptotic stability of these eigenvalues with respect to the fractional parameter that governs the operators. 

Fractional (or integro-differential) operators have received much interest from the PDE community in the past decade or so; such operators arise naturally in the context of stochastic L\'evy processes with jumps and have been studied thoroughly both from the point of view of Probability and Analysis as they proved to be accurate models to describe different phenomena in physics, finance, image processing, or ecology; see for instance \cite{Ap04,CT16,ST94} and references therein. For the mathematical background from the PDE perspective taken in this paper, see for instance \cite{BV} or \cite{Ga}. 

In this manuscript we will be interested in a class of integro-differential operators defined as follows: given a fractional parameter $s\in(0,1)$ and a Young function $G$ with $g=G'$ (see Section \ref{prel} below for this and forthcoming definitions), the fractional $g-$Laplacian operator is defined as
$$
(-\Delta_g)^s u :=\text{p.v.} (1-s)\int_{\R^n} g(|D_s u|) \frac{D_s u}{|D_s u|} {\frac{dy}{|x-y|^{n+s}}},
$$
where $D_s u (x,y):= \frac{u(x)-u(y)}{|x-y|^s}$ and p.v. stands for principal value, needed due to the singularity of the integrand. This operator, as well as the natural fractional Orlicz-Sobolev spaces where to define it, were introduced and studied in \cite{FBS} (see also \cite{ACPS}). 

This fractional $g-$Laplacian, $(-\Delta_g)^s$ is the natural operator to consider when studying non-local problems with a behavior more general than a power, and it has shown to be of much interest in the recent years, see for instance \cite{FBSV2,GKS, MSV}. In particular, the articles  \cite{BS,S,SV, FBSV3} deal with the eigenvalue problem associated to $(-\Delta_g)^s$ for $s\in(0,1)$ in a bounded domain $\Omega\subset \R^n$, i.e.
\begin{align}  \label{autov.nh}
\begin{cases}
(-\Delta_g)^s u =   \lam^s \, g\left(|u|\right) \frac{u}{|u|} &\text{ in } \Omega\\
u=0 &\text{ in } \R^n \setminus \Omega.
\end{cases}
\end{align}
(The superscript $s$ in $\lam^s$ refers to the order of differentiation, it is \emph{not} a power.)

The natural variational approach to study problems as \eqref{autov.nh} is to consider critical points of the non-homogeneous Rayleigh type quotient
\begin{equation}\label{eq.mini}
\frac{\displaystyle (1-s)\iint_{\R^n\times\R^n} G(|D_s u|)\,d\mu}{\displaystyle \int_\Omega G(|u|)\,dx},
\end{equation}
where $d\mu = |x-y|^{-n}\,dxdy$.

One of the main difficulties when dealing with eigenvalue problems in Orlicz-Sobolev spaces is precisely the lack of homogeneity and, as a consequence, the fact that multiplication by constants is not a closed operation (see for instance \cite{BS, GHLMS,GM, S, SV} and compare with \cite{FBPS, L, L2, P} for the homogeneous counterpart). This leads to the loss of the simplicity property of eigenvalues, and in general, an eigenfunction times a constant may not be an eigenfunction. In particular, eigenvalues in \eqref{eq.mini} may not be simple and they highly depend on the energy level of the normalization, $\int_\Omega G(|u|)\, dx = \alpha$; different normalizations may lead to different eigenvalues which are not easily compared. The behavior with respect to $\alpha$ for such problems is in fact an intriguing open question even for local operators.

In order to recover some of these spectral properties of homogeneous operators such as the fractional $p-$Laplacian, Franzina and Lindqvist proposed in \cite{FrLi} minimizing a quotient of norms instead of one involving modulars for the case of the variable exponent Sobolev spaces $W^{1,p(x)}_0(\Omega)$. 

In this manuscript we propose to study the following nonlocal ($1-$homogeneous) Rayleigh-type quotient:
\begin{equation}\label{eq.min}
\J_s(u):=\begin{cases}
\displaystyle\frac{[u]_{s, G}}{\|u\|_G} & \text{if } s\in (0,1)\\
\\
\displaystyle\frac{\|\nabla u\|_G}{\|u\|_G} & \text{if } s=1,
\end{cases}
\end{equation}
where $\|\cdot\|_G$ and $[\,\cdot\,]_{s,G}$ stand for the Luxemburg norm in $L^G(\Omega)$ and seminorm in $W^{s,G}_0(\Omega)$, respectively. See Section \ref{prel} for a precise definition.

Notice that the homogeneity of the quotient in \eqref{eq.min} implies that  eigenvalues are  independent of the energy level chosen to perform the minimization (in particular, we can and often will assume $\|u\|_G=1$). 

In fact, one of our main result (Theorem \ref{teoLS}) is that a sequence of eigenvalues given by critical points of $\J_s$ actually exists. 

The choice of the quotient that makes the problem homogeneous derives in an eigenvalue problem for an operator different than the usual fractional $g-$Laplacian studied in the upper cited papers. Here, we are lead to consider the non-local non-standard growth homogeneous operator $\LL_s\colon W^{s,G}_0(\Omega) \to \left(W^{s,G}_0(\Omega)\right)^\ast$ defined as 
\begin{equation}\label{Ls}
\LL_s u :=
\begin{cases}
\displaystyle (1-s) \int_{\R^n} g\left(\frac{|D_s u|}{[u]_{s,G}}\right) \frac{D_s u}{|D_s u|}  \frac{dy}{|x-y|^{n+s}} &\quad \text{ when } s\in (0,1),\\
 \, & \\
\displaystyle-\diver\left( g\left(\frac{|\nabla u|}{\|\nabla u\|_G}\right) \frac{\nabla u}{|\nabla u|} \right) &\quad \text{ when } s=1
\end{cases}
\end{equation}
and, given a bounded Lipschitz domain $\Omega\subset \R^n$, we consider the homogeneous counterpart of \eqref{autov.nh} given by 
\begin{align} \label{eigen} 
\begin{cases}
\LL_s u =   \lam^s\, g\left(\frac{|u|}{\|u\|_G} \right) \frac{u}{|u|} &\text{ in } \Omega\\
u=0 &\text{ in } \R^n \setminus \Omega.
\end{cases}
\end{align}
Therefore, we say that $\lam^s$ is an eigenvalue of \eqref{eigen} with eigenfunction $u\in W^{s,G}_0(\Omega)$ if $u$ is a nontrivial weak solution to \eqref{eigen}. We prove in Theorem \ref{teo.EL.eq} that critical points of $\J_s$ are precisely the (weak) solutions of \eqref{eigen} and hence eigenvalues of \eqref{eigen} are the critical values of $\J_s$.

A further interesting question when dealing with fractional operators and/or spaces is the asymptotic behavior as the fractional parameter $s$ approaches 1; this was the main point of interest in the seminal paper by Bourgain, Brezis and Mironescu \cite{BBM} and was one of the issues tackled in \cite{FBS}, and has been a widely studied in different scenarios, see for instance \cite{ ACPS1,BBM1, FBS1,Ponce}.

In this direction, we prove stability of solutions of equations involving the operator $\LL_s$ as the fractional parameter $s$ goes to 1 (see Theorem \ref{main}). More precisely,  if we assume that $u_s \in W^{s,G}_0(\Omega)$ is a weak solution of the problem
\begin{align*} 
\begin{cases}
\LL_s u = f_s(x,\frac{u}{\|u\|_G}) & \text{ in } \Omega\\
u=0 &\text{ in } \R^n \setminus \Omega,
\end{cases}
\end{align*}
with $f_s(x,z)$ satisfying a suitable growth behavior on $|z|$ and converging to some $f=f(x,z)$, we show that any  accumulation point $u$ of the family $\{u_s\}_{s}$ in the $L^G(\Omega)-$topology (as $s\uparrow 1$) verifies that $u\in W^{1,G}_0(\Omega)$ and it is a weak solution of the limit equation 
\begin{align*} 
\begin{cases}
\bar \LL_1 u = f(x,\frac{u}{\|u\|_g}) & \text{ in } \Omega\\
u=0 &\text{ on } \partial \Omega
\end{cases}
\end{align*}
where $\bar \LL_1$ is the operator defined \eqref{Ls} with $g$ replaced by $\bar g = \bar G'$ and $\bar G$ is a suitable limit Young function (equivalent to $G$).  See \eqref{bG} for the precise definition of $\bar G$.

As a consequence, this enables us to deal with the stability of Dirichlet eigenvalues of $\LL_s$: if $\lam^s$ is  an eigenvalue of \eqref{eigen} and $\lambda^1$ is an accumulation point of the family $\{\lam^s\}_{s>0}$ when $s\uparrow 1$, then in Corollary \ref{corolario} we prove that    $\lam^1$ is an eigenvalue of the limit equation
\begin{equation}\label{eigen1} 
\begin{cases}
\bar \LL_1 u =   \lam^1\, g\left(\frac{|u|}{\|u\|_G} \right) \frac{u}{|u|} &\text{ in } \Omega\\
u=0 &\text{ on } \partial\Omega.
\end{cases}
\end{equation}

To complete the stability results we refine this result on eigenvalues and analyze the behavior of each variational eigenvalue $\lam_k^s$ when the fractional parameter $s\uparrow 1$. The techniques needed for this analysis involve some Gamma-convergence results for the corresponding Rayleigh quotients (see Theorem \ref{thm.abst}) and we are able to show that $\lam_k^s\to \lam_k^1$ as $s\uparrow 1$ where the $\lam_k^s$ is the $k^{th}$ variational eigenvalue of \eqref{eigen} and $\lam_k^1$ is the $k^{th}$ variational eigenvalue of \eqref{eigen1}. This is the content of Corollary \ref{cor.eigen}.

A technical but somewhat interesting remark on this theorem is the following: our hypotheses in Lemma \ref{lema.desigualdad} would constraint the cases in which the existence of a sequence eigenvalues for local operators ($s=1$) can be obtained; indeed in that lemma we require $G(\sqrt{t})$ to be convex for $s=1$, whereas in the fractional scenario we include the case where $g'$ is decreasing as well, which is an exclusive (although not complementary) requirement. However, the stability result ensures that taking the limit over the existent fractional eigenvalues we get the existence local eigenvalues, thus bypassing this technical difficulty. 

\bigskip

The paper is organized as follows: in Section \ref{prel} we give some preliminary definitions. The Euler-Lagrange equation (Theorem \ref{teo.EL.eq}) is derived in Section \ref{ELeqn}. Section \ref{Exist} is devoted to the proof of the existence of a sequence of variational eigenvalues, that is Theorem \ref{teoLS}. Finally, we dedicate Section \ref{Stab} to prove the stability results.  

\section{Preliminaries}\label{prel}

In this section we present some preliminary definitions needed for the rest of the paper. 

\subsection{Young functions}

Throughout this article $G$ will denote a so-called Young function, that is, an application from $[0,\infty)$ into itself such that $G'(t)=g(t)$ where $g$ is a right-continuous function defined on $[0,\infty)$ and which is positive in $(0,\infty)$, vanishes at zero, is non-decreasing on $(0,\infty)$, and diverges when $t\to\infty$. 

In fact, we will further assume that Young functions satisfy the following growth condition: there exist constants  $1<p^-<p^+<\infty$
\begin{equation}\label{L}
p^--1\leq \frac{tg'(t)}{g(t)}\leq p^+-1,\quad t>0.
\end{equation}
From this inequality we immediately get
\begin{equation}\label{delta2}
p^- \leq \frac{tg(t)}{G(t)} \leq p^+,
\end{equation}
and these two inequalities together in turn give 
\begin{equation} \label{prop1}
p^- (p^--1) G(t) \leq t^2 g'(t) \leq p^+ (p^+ -1) G(t),
\end{equation}
which will be used later. 

Further, condition \eqref{delta2} ensures that both a Young function $G$ and its complementary function $\tilde G$ defined as 
$$
\tilde G(t):=\sup\{ tw-G(w)\colon w>0\}
$$
satisfy the so-called $\Delta_2$ condition, i.e., 
$$
G(2t)\leq C G(t), \qquad \tilde G(2t) \leq \tilde C \tilde G(t), \qquad t\geq 0.
$$
We note also that the \emph{Young inequality} 
\[
tw\leq G(t)+\tilde{G}(w)
\]
holds for $t,w>0$. It can further be shown that equality holds when $w=g(t)$. From this, the following inequality is easy to show: 
\begin{equation} \label{lemita}
\tilde G(g(t)) \leq (p^+-1) G(t).
\end{equation}

Finally, a structural (technical) condition is needed to ensure the uniform monotonicity of the funcitonals, needed for proving existence: either
$$
G(\sqrt{t}) \text{ is a convex function} \quad \text{ or }\quad g'(t)  \text{ is a decreasing function}.
$$
Observe that in the case of powers, i.e., $G(t)=t^p$, these conditions correspond to  $p\geq 2$ or $1<p<2$, respectively. For a general Young function, although these conditions are exclusive, they are not complementary. This issue is left as a rather subtle open question.

\subsection{Fractional Orlicz-Sobolev spaces}

The functional setting in this manuscript will be the suitable fractional Orlicz-Sobolev spaces $W^{s,G}_0(\Omega)$ which were recently defined by the first and second authors, see \cite{FBS}.  In this context, we have the \emph{modulars} 
\begin{align*}
\Phi_{G}(u) & := \int_\Omega G(|u|)\:dx \\
\Phi_{s,G}(u) & :=(1-s)\iint_{\R^n\times\R^n} G\left(|D_su(x,y)|\right)\:d\mu,\: s\in(0,1) \\
\Phi_{1,G}(u) & :=\int_\Omega G\left(|\nabla u|\right)\:dx
\end{align*}
with 
$$
D_s u(x,y) := \frac{u(x)-u(y)}{|x-y|^s}\quad \text{and}\quad d\mu:=\frac{dxdy}{|x-y|^n}.
$$
Moreover, $\|\cdot\|_G$ will denote the \emph{Luxemburg norm}
\[
\|u\|_G:=\inf\left\{\tau>0\colon \Phi_G\left(\frac{u}{\tau}\right)\leq 1\right\}
\]
and the \emph{Gagliardo seminorm} $[\,\cdot\,]_{s,G}$ is given by
\[
[u]_{s,G}:=\inf\left\{\tau>0\colon \Phi_{s,G}\left(\frac{u}{\tau}\right)\leq 1\right\}.
\]
We refer the reader to \cite{FBS} for the common structural properties of such spaces (see also \cite{ACPS}).

\section{The Euler-Lagrange equation}\label{ELeqn} 

In this section we derive the Euler-Lagrange equation satisfied by the critical points of $\mathcal{J}_s$.
We introduce some notations that will be helpful in this section and in the remaining of the paper. Given an Orlicz function $G$, and a fractional parameter $s\in (0,1]$, we define the functionals
\begin{align}
\label{I}\I(u) &= \|u\|_G,\\
\label{H}\HH_s(u) &= \begin{cases}
[u]_{s, G} & \text{if } s\in (0,1)\\
\|\nabla u\|_G & \text{if } s=1.
\end{cases}
\end{align}
Observe that the (homogeneous) Rayleigh quotient $\J_s$ can then be written as
$$
\J_s(u) = \frac{\HH_s(u)}{\I(u)}.
$$

The main result of this section is the following
\begin{thm} \label{teo.EL.eq}
Let $G$ be an Orlicz function satisfying \eqref{delta2} and let $s\in (0,1]$ be a fractional parameter. Then $u\in W^{s,G}_0(\Omega)$ is a critical point of $\mathcal{J}_s$ if and only if $u$ is a weak solution to the Euler-Lagrange equation \eqref{eigen}.
\end{thm}

We first give some non-rigorous heuristic: for $s\in (0,1]$ let $u$ be critical point of $\mathcal{J}_s$ and define $w_\ve:=u+\ve v$ for $v\in C^\infty_c(\Omega)$. We compute
\begin{align*}
0=\left.\frac{d}{d\varepsilon}\J_s(w_\ve)\right|_{\varepsilon=0} & = \left.\frac{d}{d\varepsilon}\frac{\mathcal{H}_s(w_\ve)}{\mathcal{I}(w_\ve)}\right|_{\varepsilon=0} \\
&= 
\left.\frac{ \mathcal{I}(w_\ve)\langle \mathcal{H}_s'(w_\ve),v\rangle -\mathcal{H}_s(w_\ve) \langle \mathcal{I}'(w_\ve),v\rangle}{(\mathcal{I}(w_\ve))^2}\right|_{\varepsilon=0} \\
&= 
 \frac{ \mathcal{I}(u)\langle \mathcal{H}_s'(u),v\rangle -\mathcal{H}_s(u) \langle \mathcal{I}'(u),v\rangle}{(\mathcal{I}(u))^2} 
\end{align*}
so the necessary condition is 
\begin{equation}\label{eq.el}
\frac{\langle \mathcal{H}_s'(u),v\rangle }{\mathcal{H}_s(u)}  =
\frac{\langle \mathcal{I}'(u),v\rangle }{\mathcal{I}(u)}.
\end{equation}

To formalize \eqref{eq.el}  we have to show that $\mathcal{H}_s$, $s\in (0,1]$ and $\mathcal{I}$ are Fr\'echet differentiable and compute their derivatives.
 
\begin{prop} \label{propo}
Let $G$ be an Orlicz function satisfying \eqref{delta2} and $s\in (0,1]$ be a fractional parameter. Let $\I\colon L^G(\Omega)\to\R$ and $\HH_s\colon W^{s,G}_0(\Omega)\to \R$  be the functionals defined in \eqref{I} and \eqref{H} resepectively. Then $\I$ and $\HH_s$ are differentiable away from 0 and their derivatives are given by
\begin{align} \label{Iprima}
\langle \mathcal{I}' (u),v\rangle &= \int_\Omega  g\left(\frac{|u|}{\|u\|_G}\right)\frac{uv}{|u|}\:dx,\\
\label{Hprima}\langle \mathcal{H}'_s (u),v\rangle &=  \begin{cases} 
\displaystyle (1-s) \iint_{\R^n\times\R^n} g\left(\frac{|D_su|}{[ u]_{s,G}}\right)\frac{D_s u}{|D_s u|}D_s v \:d\mu, & \text{if } s\in (0,1)\\
\\
\displaystyle\int_\Omega  g\left(\frac{|\nabla u|}{\|\nabla u\|_G}\right)\frac{\nabla u\cdot\nabla v}{|\nabla u|}\:dx, & \text{if } s=1
\end{cases}
\end{align}
\end{prop}
 
\begin{proof}
Let us deal first with the case $s=1$.

Observe that for $a,b\in\R$ we have that
\begin{align} \label{igualdad}
\begin{split}
G(|b|)-G(|a|) & = \int_0^1\frac{d}{dt}G\left(|t(b-a)+a|\right)\:dt \\
			  & =(b-a)\int_0^1g\left(|t(b-a)+a|\right)\frac{t(b-a)+a}{|t(b-a)+a|}\:dt.
\end{split}
\end{align}

Moreover, by definition of the Luxemburg norm  (see \cite{KR}), for all $\ve\geq 0$ 
\begin{equation}\label{eq.normas}
\int_\Omega G\left(\frac{|\nabla w_\ve|}{ \mathcal{H}_1(w_\ve) }\right)\:dx=1,  \qquad\int_\Omega G\left(\frac{|w_\ve|}{ \mathcal{I}(w_\ve)}\right)\:dx=1
\end{equation}

Let us first prove the expression for $\langle \mathcal{I}'(u),v \rangle$.

Using this with $b:=\frac{w_\ve}{\mathcal{I}(w_\ve)}$ and $a:=\frac{u}{ \mathcal{I}(u)}$ together with \eqref{eq.normas} we can compute
\begin{align} \label{eqq1}
\begin{split}
0&=\int_\Omega G\left(\frac{|w_\ve|}{\mathcal{I}(w_\ve)}\right)\:dx-\int_\Omega G\left(\frac{|u|}{\mathcal{I}(u)}\right)\:dx \\
 &=\int_\Omega \left(\frac{w_\ve}{\mathcal{I}(w_\ve)}-\frac{u}{\mathcal{I}(u)}\right)A(x,\varepsilon)\:dx
\end{split}
\end{align}
with
\[
A(x,\varepsilon):=\int_0^1 g\left(\left|t\left(\frac{w_\ve}{\mathcal{I}(w_\ve)}-\frac{u}{\mathcal{I}(u)}\right)+\frac{u}{\mathcal{I}(u)}\right|\right)\frac{t\left(\frac{w_\ve}{\mathcal{I}(w_\ve)}-\frac{u}{\mathcal{I}(u)}\right)+\frac{u}{\mathcal{I}(u)}}{\left|t\left(\frac{w_\ve}{\mathcal{I}(w_\ve)}-\frac{u}{\mathcal{I}(u)}\right)+\frac{u}{\mathcal{I}(u)}\right|}\:dt.
\]
From \eqref{eqq1} and the definition of $v_\ve$  we get
\[
\int_\Omega \left(\frac{u}{\mathcal{I}(w_\ve)}-\frac{u}{\mathcal{I}(u)}\right)A(x,\varepsilon)\:dx = -\varepsilon\int_\Omega \frac{\varphi}{\mathcal{I}(w_\ve)} A(x,\varepsilon)\:dx
\]
from where
\begin{equation}\label{eq.ecua}
\int_\Omega \left(\frac{\mathcal{I}(w_\ve)-\mathcal{I}(u)}{\varepsilon}\right)\frac{u}{\mathcal{I}(u)}A(x,\varepsilon)\:dx = \int_\Omega \varphi A(x,\varepsilon)\:dx.
\end{equation}
Now, since $v_\ve\longrightarrow u$ a.e. as $\varepsilon\rightarrow0^+$, we deduce that
\[
A(x,\varepsilon)\longrightarrow g\left(\frac{|u|}{\mathcal{I}(u)}\right)\frac{u}{|u|}\quad \text{a.e. as }\varepsilon\rightarrow0^+
\]
so \eqref{eq.ecua} implies 
\[
\frac{\langle \mathcal{I}'(u),v \rangle}{\mathcal{I}(u)}\int_\Omega g\left(\frac{|u|}{\mathcal{I}(u)}\right)|u|\:dx = \int_\Omega   g\left(\frac{|u|}{\mathcal{I}(u)}\right)\frac{uv}{|u|}\:dx
\]
that is
\[
\frac{\langle \mathcal{I}'(u),v \rangle}{\mathcal{I}(u)}=\frac{ \displaystyle \int_\Omega   g\left(\frac{|u|}{\|u\|_G}\right)\frac{uv}{|u|}\:dx}{\displaystyle\int_\Omega g\left(\frac{|u|}{\|u\|_G}\right)|u|\:dx}.
\]

We can reason similarly to get the expression for $\mathcal{H}'$.
Indeed, 
\begin{align}\label{eqq2}
\begin{split}
0&=\int_\Omega G\left(\frac{|\nabla w_\ve|}{ \mathcal{H}_1(w_\ve)}\right)\:dx-\int_\Omega G\left(\frac{|\nabla u|}{ \mathcal{H}_1(u)}\right)\:dx \\
 &=\int_\Omega \left(\frac{|\nabla u+\ve\nabla v|}{ \mathcal{H}_1(w_\ve)}-\frac{|\nabla u|}{ \mathcal{H}_1(u)}\right)B(x,\varepsilon)\:dx
\end{split}
\end{align}
with
\[
B(x,\varepsilon):=\int_0^1 g\left(\left|t\left(\frac{|\nabla w_\ve|}{ \mathcal{H}_1(w_\ve)}-\frac{|\nabla u|}{ \mathcal{H}_1(u)}\right)+\frac{|\nabla u|}{ \mathcal{H}_1(u)}\right|\right)\frac{t\left(\frac{|\nabla w_\ve|}{ \mathcal{H}_1(w_\ve)}-\frac{|\nabla u|}{ \mathcal{H}_1(u)}\right)+\frac{|\nabla u|}{ \mathcal{H}_1(u)}}{\left|t\left(\frac{|\nabla  w_\ve|}{ \mathcal{H}_1(w_\ve)}-\frac{|\nabla u|}{ \mathcal{H}_1(u)}\right)+\frac{|\nabla u|}{ \mathcal{H}_1(u)}\right|}\:dt.
\]
From \eqref{eqq2} we deduce that
\begin{equation}\label{eq.grad}
\begin{split}
\frac{1}{ \mathcal{H}_1(u)}\int_\Omega \left(\frac{|\nabla u+\varepsilon\nabla v|-|\nabla u|}{\varepsilon}\right)B(x,\varepsilon)\:dx =\\ 
\int_\Omega \frac{|\nabla u+\varepsilon\nabla v|}{\varepsilon}\left(\frac{1}{ \mathcal{H}_1(u)}-\frac{1}{ \mathcal{H}_1(w_\ve)}\right)B(x,\varepsilon)\:dx.
\end{split}
\end{equation}
As before
\[
B(x,\varepsilon)\longrightarrow g\left(\frac{|\nabla u|}{ \mathcal{H}_1(u)}\right) \quad \text{a.e. as }\varepsilon\rightarrow0^+
\]
and 
\[
\frac{|\nabla u+\varepsilon\nabla v|-|\nabla u|}{\varepsilon}\longrightarrow\frac{\nabla u\cdot\nabla v}{|\nabla u|} \quad \text{a.e. as }\varepsilon\rightarrow0^+
\]
For the right hand side of \eqref{eq.grad}, observe that
\begin{align*}
&\frac{|\nabla u+\varepsilon\nabla v|}{\varepsilon}\left(\frac{1}{ \mathcal{H}_1(u)}-\frac{1}{ \mathcal{H}_1(w_\ve)}\right)=\\
&\frac{|\nabla u+\varepsilon\nabla v|}{\mathcal{H}_1(w_\ve) \mathcal{H}_1(u)}\left(\frac{ \mathcal{H}_1(w_\ve)- \mathcal{H}_1(u)}{\varepsilon}\right)\longrightarrow \frac{|\nabla u|}{( \mathcal{H}_1(u))^2} \langle  \mathcal{H}_1'(u) , v \rangle
\end{align*}
a.e. as $\varepsilon\rightarrow0^+$ so taking limits in \eqref{eq.grad} we get
\[
\frac{1}{ \mathcal{H}_1(u)}\int_\Omega \frac{\nabla u\cdot\nabla \varphi}{|\nabla u|}g\left(\frac{|\nabla u|}{ \mathcal{H}_1(u)}\right)\:dx = \int_\Omega \frac{|\nabla u|}{( \mathcal{H}_1(u))^2}\langle  \mathcal{H}_1'(u) , v \rangle g\left(\frac{|\nabla u|}{ \mathcal{H}_1(u)}\right)\:dx
\]
from where
\begin{equation*}
\frac{\langle  \mathcal{H}_1'(u) , v \rangle}{ \mathcal{H}_1(u)}=\frac{ \displaystyle\int_\Omega  g\left(\frac{|\nabla u|}{\|\nabla u\|_G}\right)\frac{\nabla u\cdot\nabla v}{|\nabla u|}\:dx}{ \displaystyle\int_\Omega g\left(\frac{|\nabla u|}{\|\nabla u\|_G}\right)|\nabla u|\:dx}.
\end{equation*}

Finally, notice that replacing in \eqref{igualdad} $b$ and $a$ by $D_su$ and $D_s w_\ve$, respectively, and using that
$$
\iint_{\R^n\times \R^n} G\left(\frac{|D_s w_\ve|}{ \mathcal{H}_s(w_\ve)}\right)\:d\mu=1,
$$
by repeating the previous steps we get
\[
\langle  \mathcal{H}'_s(u) , v \rangle = \iint_{\R^n\times\R^n} g\left(\frac{|D_su|}{[ u]_{s,G}}\right)\frac{u(x)-u(y)}{|u(x)-u(y)|}D_s v \:d\mu \qquad \forall v \in W^{s,G}_0(\Omega).
\]
The proof is now completed.
\end{proof}

\begin{rem}
Observe that from the previous Proposition, $\mathcal{H}_s' = \LL_s$ for $s\in (0,1]$.
\end{rem}

We are now in position to give the 
\begin{proof}[Proof of Theorem \ref{teo.EL.eq}]
Let $u$ be a critical point of $\mathcal{J}_1$. Gathering \eqref{eq.el}, \eqref{Iprima} and \eqref{Hprima}  we get
$$
\int_\Omega  g\left(\frac{|\nabla u|}{\|\nabla u\|_G}\right)\frac{\nabla u\cdot\nabla v}{|\nabla u|}\:dx = \lam^1 \int_\Omega  g\left(\frac{|u|}{\|u\|_G}\right)\frac{uv}{|u|}\:dx \quad \forall v\in W^{1,G}_0(\Omega),
$$
giving the result by a density argument.

An analogous expression  for $\lam^s$ is obtained in a similar fashion. 
\end{proof}

\section{Existence of a sequence of variational eigenvalues}\label{Exist}

This section is devoted to the proof of a general existence result for variational eigenvalues, namely Theorem \ref{teoLS}. The proof is based on the Ljusternik-Schnirelman theory, although we point out that existence of the first variational eigenvalue $\lam_1^s$, $s\in (0,1]$ could be easily achieved by the direct method of the Calculus of Variations. 

The main result in this section reads as follows:

\begin{thm} \label{teoLS}
Let $G$ be a Young function satisfying \eqref{L}. 
Assume one of the following scenarios:

\begin{itemize}
\item[(i)] $s\in (0,1)$ and $G$ additionally satisfies that $G(\sqrt{t})$ is convex or $g'$ is decreasing;
\item[(ii)] $s=1$ and $G(\sqrt{t})$ is convex.
\end{itemize}

Then there exists a sequence $\{\lam_k^s\}_{k\in\N}$ of critical points of $\J_s$, $\lam_k^s\to \infty$ as $k\to\infty$. Moreover, these critical points have the following variational characterization
\begin{equation} \label{minmax}
\lam_k^s = \inf_{K\in \mathcal{C}_k} \sup_{u\in K} \mathcal{H}_s(u)
\end{equation}
where, for any $k\in\N$,
$$
\mathcal{C}_k := \{ K\subset M \text{ compact, symmetric with } \mathcal{H}_s(u)>0 \text{ on } K  \text{ and } \gamma(K)\geq k\},
$$
$$
M := \{ u\in W^{s,G}_0(\Omega) \colon \mathcal{I}(u)=1\}
$$
and   $\gamma$ is the Krasnoselskii genus of $K$.
\end{thm}

\begin{rem}
See \cite{R} for the definition and properties of $\gamma$.
\end{rem}

%
%

A first fundamental ingredient for the proof of Theorem \ref{teoLS} is the uniform monotonicity of the operator $\LL_s$, which can be deduced from the following key lemma.

\begin{lem} \label{lema.desigualdad}
Let $G$ be a Young function satisfying \eqref{L}. Then, for any $a,b\in \R$ there exists a positive constant $C$ independent of $a$ and $b$ such that
\begin{align*}
\left( g(|b|)\frac{b}{|b|}-  g(|a|)\frac{a}{|a|} \right)(b-a)  \geq 
\begin{cases}
CG(|b-a|) & \text{ if  $G(\sqrt{t})$ is convex},\\
C (b-a)^2 g'(|b-a|) & \text{ if $g'$ is decreasing}.
\end{cases}
\end{align*}
 
 Moreover, when $G(\sqrt{t})$ is convex, the first inequality holds for $a,b\in \R^n$.

\end{lem} 

\begin{proof}
Denote 
\[
I_{a,b}:=\left( g(|b|)\frac{b}{|b|}-  g(|a|)\frac{a}{|a|} \right).
\] 
We first assume that $G(\sqrt{t})$ is convex. The result in this case for $a,b \in \R^n$ can be found in \cite[Lemma 3.1]{CSS}, nevertheless, we put forward here a simple proof for $a,b\in \R$ based on the convexity of $G$. Indeed, 
$$
G(|b|) \leq G\left(\left| \frac{a+b}{2}  \right| \right) + g(|b|)\frac{b}{|b|}\frac{b-a}{2}, \qquad 
G(|a|) \leq G\left(\left| \frac{a+b}{2}  \right| \right) + g(|a|)\frac{a}{|a|}\frac{a-b}{2}.
$$
Adding the above two relations we find that
$$
\frac{1}{2}\left( g(|b|)\frac{b}{|b|}-  g(|a|)\frac{a}{|a|} \right)(b-a) \geq  G(|b|) + G(|a|)-2G\left( \left| \frac{a+b}{2}\right|\right).
$$
Then, by using \cite[Lemma 2.1]{Lamperti} it follows that
$$
I_{a,b}(b-a) \geq    4G\left( \left| \frac{b-a}{2} \right|\right) \geq    2^{2-p^+}G\left( |b-a|\right).
$$

Assume now that $g'$ is decreasing. Without loss of generality we can take  $b>a$. If $a,b>0$ a straightforward computation gives that
\[
I_{a,b} = g(b)-g(a) = \int_a^b g'(t)\,dt \geq (b-a)g'(b)\geq (b-a)g'(b)
\]
so 
\[
I_{a,b}(b-a)\geq (b-a)^2g'(b-a)
\]
and the result holds with $C=1$. 

If $a=0$, by using \eqref{L} we get
$$
I_{a,b}b = g(b)b \geq \frac{1}{p^+-1} b^2 g'(b).
$$

Finally, if $b>0$ and $a<0$, then $b=|b|$, $a=-|a|$. In this case, using \eqref{delta2} and \eqref{L} we get
\begin{align*}
I_{a,b}= (g(|b|)+g(|a|))(|a|+|b|) & \geq C g(|a|+|b|)(|a|+|b|)\\
&\geq \frac{C}{p^+} (|a|+|b|)^2 g'(|a|+|b|)\\
&=
\frac{C}{p^+} (b-a)^2 g'(b-a),
\end{align*}
which concludes the proof.
\end{proof}

As a direct consequence of the previous lemma we get the desired monotonicity property:

\begin{prop} \label{propo.mono}
Let $G$ be a Young function satisfying \eqref{L}. Then
\begin{itemize}
\item[(i)] If  $G(\sqrt{t})$ is convex
\begin{align*}
\left\langle \mathcal{H}_s'(u) -\mathcal{H}_s'(v),\frac{u}{[u]_{s,G}}-\frac{v}{[v]_{s,G}} \right\rangle  &\geq 
C\Phi_{s,G}\left(\left|\frac{u}{[u]_{s,G}}-\frac{v}{[v]_{s,G}} \right|\right),\\
\left\langle \mathcal{H}_1'(u) -\mathcal{H}_1'(v),\frac{u}{\|\nabla u \|_G}-\frac{v}{\|\nabla v\|_G} \right\rangle  &\geq 
C\Phi_{G}\left(\left|\frac{\nabla u}{\|\nabla u\|_G}-\frac{\nabla v}{\|\nabla v\|_G} \right|\right).
\end{align*}

\item[(ii)] If $g'$ is decreasing
\begin{align*}
\Big\langle \mathcal{H}_s'(u)& -\mathcal{H}_s'(v),\frac{u}{[u]_{s,G}}-\frac{v}{[v]_{s,G}} \Big\rangle   \geq \\
&C\iint_{\R^n\times\R^n} \left(\frac{|D_s u|}{[u]_{s,G}} +\frac{|D_s v|}{[v]_{s,g}} \right)^2 g'\left(\left|D_s \left(\frac{u}{[u]_{s,G}}-\frac{v}{[v]_{s,G}}\right) \right| \right)\,d\mu.
\end{align*}
\end{itemize}
\end{prop}

\medskip
 
The following proposition gives the structural properties of $\mathcal{I}$ and $\mathcal{H}_s$ needed to apply the so-called Ljusternik-Schnirelman theory: 

\begin{prop} \label{LS}
The functionals $\mathcal{I}$ and $\mathcal{H}_s$ satisfy the following conditions:

\begin{itemize}
\item[($h_1$)] $\mathcal{I}$ and $\mathcal{H}_s$  are  $C^1(W^{s,G}_0(\Omega)\setminus\{0\},\R)$ even functionals  with $\mathcal{I}(0)=\mathcal{H}_s(0)=0$ and the level set
$$
M := \{ u\in W^{s,G}_0(\Omega) \colon \mathcal{H}_s(u)=1\}
$$
is bounded.

\item [($h_2$)] $\mathcal{I}'$ is strongly continuous, i.e.,
$$
u_k\rightharpoonup u \text{ in } W^{s,G}_0(\Omega) \implies \mathcal{I}'(u_k) \to \mathcal{I}'(u).
$$
Moreover, for any $u\in W^{s,G}_0(\Omega)$ it holds that
$$
\langle \mathcal{I}'(u),u\rangle=0 \iff \mathcal{I}(u)=0 \iff u=0.
$$

\item [($h_3$)] $\mathcal{H}_s'$ is continuous, bounded and, as $k\to\infty$, it holds that
$$
u_k\rightharpoonup u, \quad \mathcal{H}_s'(u_k)\rightharpoonup v, \quad \langle \mathcal{H}_s'(u_k),u_k\rangle\to\langle v,u\rangle \implies u_k\to u \text{ in } W^{s,G}_0(\Omega).
$$

\item [($h_4$)] For every $u \in W^{s,G}_0(\Omega) \setminus\{0\}$ it holds that
$$
\langle \mathcal{H}_s'(u),u\rangle>0,\qquad \displaystyle\lim_{t\to+\infty}\mathcal{H}_s(tu)=+\infty,\qquad \displaystyle\inf_{u\in M}\langle \mathcal{H}_s'(u),u\rangle>0.
  $$
\end{itemize}

\end{prop}

\begin{proof}
Let us check $(h_1)$--$(h_4)$.

$(h_1)$ Clearly, the maps $\mathcal{I}$ and $\mathcal{H}_s$ are even and $\mathcal{I}(0)=\mathcal{H}_s(0)=0$. The differentiability away from $0$ of $\I$ and $\HH_s$  was proved in Proposition \ref{propo}. The boundedness of the level set $M$ is a direct consequence of Poincar\'e's inequality (see for instance \cite[Theorem 2.12]{FBPLS}). 

\medskip

$(h_2)$ From \eqref{L} we get
 
  $$
  \langle \mathcal{I}'(u),u \rangle \leq \|u\|_G \int_\Omega g\left(\frac{| u|}{\|u\|_G}\right) \frac{|u|}{\|u\|_G} \,dx \leq p^+ \|u\|_G \int_\Omega G\left( \frac{|u|}{\|u\|_G} \right)\,dx =  p^+ \|u\|_G.
  $$
An analogous argument yields
$$
p^- \mathcal{I}(u) = p^- \|u\|_G \leq  \langle \mathcal{I}'(u),u \rangle \leq p^+ \|u\|_G = p^+ \mathcal{I}(u).
$$
Then immediately it follows that
$$
\langle \mathcal{I}'(u),u\rangle=0 \ \Leftrightarrow \ \mathcal{I}(u)=0 \ \Leftrightarrow \ u=0.
$$


Let us check that $\mathcal{I}'$ is strongly continuous.  Let $u_k\rightharpoonup u$ in $W^{s,G}_0(\Omega)$, we need to show that $\mathcal{I}'(u_k)\to \mathcal{I}'(u)$ in $W^{-s,\tilde G}(\Omega)$. We will achieve this by showing that any subsequence of $\mathcal{I}'(u_k)$ has a further subsequence that converges to $\mathcal{I}'(u)$. For the sake of simplicity of notation subsequences will still be denoted by the same index). Given any subsequence of $\mathcal{I}'(u_k)$ the corresponding $u_k$ converges weakly to $u$ in $W^{s,G}_0(\Omega)$ so that, up to a further subsequence, we may assume that $u_k\to u$ in $L^G(\Omega)$ and a.e.

Next, observe that
  \begin{align*} 
  \begin{split}
	|\langle \mathcal{I}'(u_{k})-\mathcal{I}'(u),v\rangle|&=\left|\int_{\Omega}\left(g\left(\frac{|u_k|}{\|u_k\|_G} \right)\frac{u_k}{|u_{k}|}-g\left(\frac{|u|}{\|u\|_G} \right)\frac{u}{|u|}\right) v\,dx\right|\\
	&\leq \left\|g\left(\frac{|u_k|}{\|u_k\|_G} \right)\frac{u_k}{|u_{k}|}-g\left(\frac{|u|}{\|u\|_G} \right)\frac{u}{|u|} \right\|_{\tilde G} \left\| v \right\|_G.
	\end{split}
\end{align*}
Since both $G$ and $\tilde G$ satisfy the $\Delta_2$ condition, the last expression goes to 0 when $k\to\infty$ if
\begin{equation} \label{xxxx}
\lim_{k\to\infty}\int_\Omega \tilde G \left( g\left(\frac{|u_k|}{\|u_k\|_G} \right) \frac{u_k}{|u_k|}     -g\left(\frac{|u|}{\|u\|_G} \right) \frac{u}{|u|} \right) \,dx =0.
\end{equation}

Since, as mentioned above, $u_k\to u$ in $L^G(\Omega)$ and $u_k\to u$ a.e. we have that there exists $h\in L^1(\Omega)$ such that $|u_k|\leq h$ a.e. in $\Omega$ (see \cite[Theorem 4.9]{brezis}). This, together with \eqref{lemita} and the $\Delta_2$ condition allows to bound \eqref{xxxx} by
$$
C\int_\Omega \left[ G\left( \frac{|u_k|}{\|u_k\|_G}\right) + G\left( \frac{|u|}{\|u\|_G}\right) \right] \,dx \leq \int_\Omega \left[ G\left( \frac{|h|}{\|u\|_G}\right) + G\left( \frac{|u|}{\|u\|_G}\right) \right] \,dx
$$
where we have used also the lower semicontinuity of the $L^G$ norm. Therefore, \eqref{xxxx} follows from the dominated convergence theorem and $ \mathcal{I}'(u_{k})$ converges to $\mathcal{I}'(u)$ as desired.

\bigskip

$(h_3)$ $\bullet$ $\mathcal{H}_s'$ is bounded:  in light of equation \eqref{lemita}
\begin{align*}
 |\langle \mathcal{H}_s'(u),v \rangle| &\leq [v]_{s,G} \iint_{\R^n\times\R^n} g\left(\frac{| D_su|}{[u]_{s,G}}\right) \frac{|D_s v|}{[v]_{s,G}} \,d\mu \\
 &\leq 2[v]_{s,G} \left\| g\left(\frac{| D_su|}{[u]_{s,G}} \right) \right\|_{\tilde G,d\mu} \left\| \frac{D_s v}{[v]_{s,G}} \right\|_{G,d\mu}\\
 &\leq 2(p^+-1)[v]_{s,G} \left\| \frac{ D_su}{[u]_{s,G}}   \right\|_{G,d\mu} \left\| \frac{D_s v}{[v]_{s,G}} \right\|_{G,d\mu}\\
 &= 2(p^+-1)
 [v]_{s,G} \left[ \frac{u}{[u]_{s,G}}\right]_{s,G} \left[ \frac{ v}{[v]_{s,G}} \right]_{s,G}\\
 &\leq 2(p^+-1) [v]_{s,G}
\end{align*}
from where 
$$
\|\mathcal{H}_s'(u)\|_{-s,\tilde G} = \sup \left\{ \frac{|\langle \mathcal{H}_s'(u),v \rangle|}{[v]_{s,G}} \colon v\in W^{s,G}(\R^n),v\neq0 \right\}
$$
is bounded. Notice further that that the bound is independent of $u$.

\bigskip

$\bullet$ $\mathcal{H}_s'$ is continuous: let $\{u_k\}_{k\in\N}\subset W^{s,G}_0(\Omega)$ be such that $u_k \to u$ in $W^{s,G}_0(\Omega)$.

By using H\"older's inequality 
\begin{align*}
 |\langle \mathcal{H}_s'(u_k)-\mathcal{H}_s'(u),v \rangle| &= \left| \iint_{\R^n\times\R^n} \left( g\left( \frac{|D_s u_k|}{[u_k]_{s,G}} \right) \frac{D_s u_k}{|D_s u_k|} -  g\left( \frac{|D_s u|}{[u_k]_{s,G}} \right) \frac{D_s u}{|D_s u|}  \right) D_s v \,d\mu \right| \\
 &\leq 
 \left\| g\left( \frac{|D_s u_k|}{[u_k]_{s,G}} \right) \frac{D_s u_k}{|D_s u_k|} -  g\left( \frac{|D_s u|}{[u_k]_{s,G}} \right) \frac{D_s u}{|D_s u|} \right\|_{\tilde G,d\mu} [v]_{s,G}.
\end{align*}
The proof from here proceeds similarly to that of the continuity of $\mathcal{I}'$ after Equation \eqref{xxxx}. 

\bigskip

$\bullet$ It remains to be shown that if $\{u_k\}_{k\in\mathbb{N}}$ is a sequence in $W^{s,G}_0(\Omega)$ such that
\begin{equation} \label{asump}
  u_k\rightharpoonup u \, \text{ in } W^{s,G}_0(\R^n), \quad \mathcal{H}_s'(u_k)\rightharpoonup v \, \text{ in } W^{-s,\tilde G}_0(\R^n), \quad \langle \mathcal{H}_s'(u_k),u_k\rangle\to\langle v,u\rangle\
\end{equation}
then $u_k \to u$ in $W^{s,G}_0(\Omega)$.

If we assume that $G(\sqrt{t})$ is convex, from Proposition \ref{propo.mono} we get 
$$
I_k:=\left\langle \mathcal{H}_s'(u) -\mathcal{H}_s'(u_k),\frac{u}{[u]_{s,G}}-\frac{u_k}{[u_k]_{s,G}} \right\rangle  \geq 
C\Phi_{s,G}\left(\left|\frac{u}{[u]_{s,G}}-\frac{u_k}{[u_k]_{s,G}} \right|\right).
$$
If otherwise, $g'$ is increasing, again Proposition \ref{propo.mono} together with \eqref{prop1} gives that
\begin{align*}
I_k &\geq 
C\iint_{\R^n\times\R^n} \left(\frac{|D_s u|}{[u]_{s,G}} +\frac{|D_s v|}{[v]_{s,g}} \right)^2 g'\left(\left|D_s \left(\frac{u}{[u]_{s,G}}-\frac{v}{[v]_{s,G}}\right) \right| \right)\,d\mu\\
&\geq  C p^-(p^--1)   
 \Phi_{s,G}\left(\left|\frac{u}{[u]_{s,G}}-\frac{u_k}{[u_k]_{s,G}}\right|\right).
\end{align*}
Therefore, since  modulars and norms are comparable, both whether $G(\sqrt{t})$ is convex or $g'$ is an increasing function we have that
$$
I_k  \geq 
C\left[\frac{u}{[u]_{s,G}}-\frac{u_k}{[u_k]_{s,G}}\right]_{s,G}^p
$$
for some exponent $p$ depending on $p^+$ and $p^-$.

Assume that $I_k\to 0$ as $k\to\infty$. Then, as a consequence $v_k=\frac{u_k}{[u_k]_{s,G}} \to v:=\frac{u}{[u]_{s,G}}$  in $W^{s,G}_0(\Omega)$. We claim that this implies that $u_k\to u$  in $W^{s,G}_0(\Omega)$, as desired. If $[u_k]_{s,G}\to \alpha\geq [u]_{s,G}$, since $u_k\to u$ in $L^G(\Omega)$,
$$
u_k = [u_k]_{s,G} v_k \stackrel{k\to\infty}{\implies} u=\alpha v = u \frac{\alpha}{[u]_{s,G}}
$$
from where $\alpha=[u]_{s,G}$.

Now, let us see that $I_k\to 0$ as $k\to\infty$. For this end, we split the limit as follows
\begin{align*}
\lim_{k\to\infty}I_k = & 
\lim_{k\to\infty}\left\langle \mathcal{H}_s'(u) -\mathcal{H}_s'(u_k),\frac{u}{[u]_{s,G}}- \frac{u}{[u_k]_{s,G}} \right\rangle \\ 
& + \lim_{k\to\infty} \left\langle  \mathcal{H}_s'(u) -\mathcal{H}_s'(u_k), \frac{u}{[u_k]_{s,G}}  -\frac{u_k}{[u_k]_{s,G}} \right\rangle\\
 &= (A) + (B).
\end{align*}

If the sequence $\{u_k\}_{k\in\N}$ converges to $u$ in $W^{s,G}_0(\R^n)$, there is nothing to prove, so we assume that $[u_k]_{s,G}\to \alpha \geq [u]_{s,G}$ as $k\to\infty$.

Let us deal with the first term
\begin{align*}
(A) &= \left( \frac{1}{[u]_{s,G}} - \frac{1}{\alpha}\right)\lim_{k\to\infty}\left\langle \mathcal{H}_s'(u) -\mathcal{H}_s'(u_k),u-u_k \right\rangle \\
&= \left( \frac{1}{[u]_{s,G}} - \frac{1}{\alpha}\right) 
\Big[ 
\lim_{k\to\infty}\left\langle \mathcal{H}_s'(u),u \right\rangle 
- \lim_{k\to\infty}\left\langle \mathcal{H}_s'(u_k),u \right\rangle\\
&\qquad\qquad\qquad\qquad -\lim_{k\to\infty}\left\langle \mathcal{H}_s'(u),u_k \right\rangle
+ \lim_{k\to\infty}\left\langle \mathcal{H}_s'(u_k),u_k \right\rangle \Big]\\
&:=(A_1) + (A_2) + (A_3) + (A_4).
\end{align*}

The first assumption in \eqref{asump} gives that $(A_1) + (A_3) = 0$. By the the second assumption, $(A_2)= -\langle v,u\rangle$ and by the third one, $(A_3) = \langle v,u \rangle$. As a consequence, $(A)= 0$.

For the second term  we have
\begin{align*}
(B) \cdot \alpha&= \lim_{k\to\infty} \left\langle  \mathcal{H}_s'(u) -\mathcal{H}_s'(u_k), u-u_k \right\rangle\\
&= \lim_{k\to\infty} \left\langle  \mathcal{H}_s'(u), u-u_k \right\rangle -  \lim_{k\to\infty} \left\langle  \mathcal{H}_s'(u_k), u \right\rangle + 
\lim_{k\to\infty} \left\langle  \mathcal{H}_s'(u_k), u_k \right\rangle\\
&:= (B_1) + (B_2) + (B_3).
\end{align*}
The first assumption in \eqref{asump} implies that $(B_1)= 0$.
From the second condition in \eqref{asump},  $(B_2) = -\langle v,u\rangle$, but the third assumption implies that $(B_3) =  \langle v,u\rangle$. Thus, $(B)= 0$.

\bigskip

$(h_4)$ It is clear that, for any $u\in W^{s,G}_0(\Omega)\setminus\{0\}$,
$$
\langle \mathcal{H}_s'(u),u\rangle>0,\qquad \displaystyle\lim_{t\to+\infty}\mathcal{H}_s(tu)=+\infty,\qquad \displaystyle\inf_{u\in M}\langle \mathcal{H}_s'(u),u\rangle>0.
$$
The proof is now concluded.
\end{proof}

We are now in position to prove Theorem \ref{teoLS}.

\begin{proof}[Proof of Theorem \ref{teoLS}]

With the notation of Proposition \ref{LS}, we apply the so-called Ljusternik-Schnirelman theory to the functionals $\mathcal{I}$ and $\mathcal{H}_s$ on the level sets $M$. 

In light of Proposition \ref{LS}, by \cite[Theorem 9.27]{MMP} there exist a sequence of numbers $\{\mu_k\}_{k\in\N}\searrow 0$ and functions $\{u_k\}_{k\in\N}\in W^{s,G}_0(\Omega)$ normalized such that $\mathcal{I}(u_k)=1$ for which
$$
  \langle \mathcal{H}_s'(u_k),v\rangle = \mu_k^{-1} \langle \mathcal{I}'(u_k),v\rangle \qquad \forall v\in W^{s,G}_0(\Omega).
$$
Moreover, 
$$
\mu_k = \sup_{K\in \mathcal{C}_k} \min_{u\in K} \mathcal{H}_s(u)
$$
where, for any $k\in\mathbb{N}$,
$$
\mathcal{C}_k:=\{ K\subset M \text{ compact, symmetric with } \mathcal{I}(u)>0 \text{ on } K  \text{ and } \gamma(K)\geq k\},
$$
and the $\gamma$ is the Krasnoselskii genus of $K$.

From the last expressions it follows that $\lam_k=\mu_k^{-1}$ is an eigenvalue of \eqref{eigen} with eigenfunction $u_k$; moreover, the $1-$homogeneity of equation \eqref{eigen} yields \eqref{minmax} and the proof concludes.
\end{proof}

\section{Stability results}\label{Stab}

This section is dedicated to prove some stability results for eigenvalues and eigenfunctions, when $s\uparrow1$ in \eqref{eigen}.
First we prove a general convergence result for solutions of
\begin{align} \label{eq.s}
\begin{cases}
\LL_s u = f_s(x,\frac{u}{\|u\|_G}) & \text{ in } \Omega\\
u=0 &\text{ in } \R^n \setminus \Omega,
\end{cases}
\end{align}
when $s\uparrow1$ that, in the eigenvalue case, i.e. when
$$
f_s(x,t) = \lambda^s g(t) \frac{t}{|t|},
$$
will give us the convergence of any sequence of eigenvalues (and eigenfunctions) to some eigenvalue (and eigenfunction) of the limit problem.

Then we specialize in the sequence of variational eigenvalues constructed in the previous section. We show that the $k^{th}$ eigenvalues $\lambda_k^s$ converge as $s\uparrow 1$ to the $k^{th}$ variational eigenvalue of the corresponding limit problem.

\subsection{Stability of solutions}
In order to analyze the stability of solutions to \eqref{eq.s} we impose the following hypotheses on $f$ that are standard in the literature:
\begin{enumerate}
\item[($f_1$)]   $f\colon\Omega\times  \R \to \R$ is a Carath\'eodory function, i.e. $f(\cdot,z)$ is measurable for any $z\in\R$  and $f(x,\cdot)$ is continuous a.e. $x\in \Omega$;

\item[($f_2$)] there exist a constant $C>0$ such that $|f(x,z)|\le C(1 + h(|z|))$, where $h=H'$ with $H$ a Young function such that $H\prec\prec G^*$.
\end{enumerate}

Given $f_s$  and $f$  functions satisfying ($f_1$) and ($f_2$) such that $f_s(x,z) \to f(x,z)$ uniformly on compacts sets of $z\in\R$ as $s\uparrow 1$, we show that weak solutions of \eqref{eq.s} converge as $s\uparrow 1$, in a suitable sense, to weak solutions of
\begin{align} \label{eq.1}
\begin{cases}
\bar\LL_{1} u = f(x,u) & \text{ in } \Omega\\
u=0 &\text{ on } \partial \Omega,
\end{cases}
\end{align}
In \eqref{eq.1} the differential operator $\bar\LL_1$ is given by
$$
\bar\LL_1 u := -\diver\left(\bar g\left(\frac{|\nabla u|}{\|\nabla u\|_{\bar G}}\right)\frac{\nabla u}{|\nabla u|}\right)
$$
where $\bar g = \bar G '$ and the Young function $\bar G$ (which is equivalent to $G$, see \cite{FBS}) is given by
\begin{equation}\label{bG}
\bar G(t)=\lim_{s\uparrow 1}(1-s)\int_0^1 \int_{\mathbb{S}^{n-1}} G(t|z_n| r^{1-s})\,dS_z \frac{dr}{r}.
\end{equation}
See \cite{FBS} for the appearance of $\bar G$ and some examples.

The main result of this subsection is:

\begin{thm}\label{main}
Let $\Omega\subset\R^n$ be a Lipschitz domain. Let $G$ be a Young function satisfying either that $G(\sqrt{t})$ is convex or that $g'$ is decreasing. Let $f$ and $f_s$ be functions satisfying ($f_1$) and ($f_s$) and let $0<s_k\uparrow 1$ and $u_k\in W^{s_k,G}_0(\Omega)$ be a sequence of solutions of \eqref{eq.s} such that $\sup_{k\in\N} [u_k]_{s_k,G}<\infty$. Then, any accumulation point $u$ of the sequence $\{u_k\}_{k\in\N}$ in the $L^G(\Omega)-$topology verifies that $u\in W^{1,G}_0(\Omega)$ and it is a weak solution of \eqref{eq.1}.
\end{thm}

In order to prove Theorem \ref{main} we prove some auxiliary lemmas. We begin with the  next result that gives an uniform asymptotic behavior of the functional $\mathcal{H}_s$.

\begin{lem}\label{asymptotic.development}
Let $u\in W^{1,G}_0(\Omega)$ be fixed and for any $s\in (0,1]$, let $v_s\in W^{s,G}_0(\Omega)$ be such that $[v_s]_{s,G}\leq C$ for any $s\in (0,1]$. Then
$$
\mathcal{H}_{s}(u+\varepsilon v_s) = \mathcal{H}_{s}(u) + \varepsilon\langle \LL_s u,v_s\rangle + o(\varepsilon)\text{ as }\varepsilon\rightarrow0,
$$
where $o(\varepsilon)$ depends only on $C$.
\end{lem}

\begin{proof}
It is a direct consequence of Proposition \ref{propo}.
\end{proof}

The following proposition is the key to study the behavior of sequences as $s\uparrow 1$. 
\begin{prop}\label{bbm.norma}
Let $0\le s_k\  \uparrow 1$ and $\{u_k\}_{k\in\N}\subset L^G(\R^n)$ be such that
$$
	\sup_{k\in\N} [u_k]_{s_k,G} <\infty \quad \text{ and }\quad  \sup_{k\in\N}\|u_k\|_G <\infty.
$$
Then there exists $u\in L^G(\R^n)$ and a subsequence $\{u_{k_j}\}_{j\in\N}\subset \{u_k\}_{k\in\N}$ such that $u_{k_j}\to u$ in $L^G_{loc}(\R^n)$. Moreover,  $u\in W^{1,G}(\R^n)$ and the following estimate holds
$$
\|\nabla u\|_{\bar G }\leq  \liminf_{k\to\infty} [u_k]_{s_k,G}.
$$
\end{prop} 

\begin{proof}
First, notice that bounded modulars  are equivalent to bounded norms and seminorms since we are assuming the $\Delta_2$ condition. Then, by   \cite[Theorem 5.1]{FBS}, up to a subsequence if neccesary, there exists $u\in L^G(\R^n)$ such that $u_k \to u$ in $L^G_{loc}(\R^n)$ and further 
$$
\Phi_{1, \bar  G}(u) \leq \liminf_{k\to\infty} \Phi_{s_k,G}(u_k).
$$
This implies that for any $k\in\mathbb{N}$
\[
\int_{\R^n} \bar G \left( \frac{|\nabla u|}{[u_k]_{s_k,G}} \right)\,dx \leq \liminf_{k\to\infty} (1-s_k) \iint_{\R^n\times\R^n} G\left( \frac{|D_{s_k}u_k(x,y)|}{[u_k]_{s_k,G}} \right)\,d\mu=1
\]
and by definition of the Luxemburg norm
$$
\|\nabla u\|_{\bar G   }\leq  [u_k]_{s_k,G}.
$$
Taking $\liminf$ as $k\to\infty$  we obtain the desired inequality.
\end{proof}

\begin{rem}\label{rem.BBM.fijo}
A similar argument as in the proof of Proposition \ref{bbm.norma} gives the convergence
$$
[u]_{s_k, G} \to \|\nabla u\|_{\bar G},
$$
using \cite[Theorem 4.1]{FBS} instead of \cite[Theorem 5.1]{FBS}. The details are left to the reader.
\end{rem}

\begin{lem} \label{key.lema}
Assume $\Omega$ that is Lipschitz. Let $s_k\uparrow 1$ and $v_k\in W^{s_k,G}_0(\Omega)$ be such that $\sup_{k\in\N} [v_k]_{s_k,G}<\infty$. Assume, without loss of generality, that $v_k\to v$ strongly in $L^G(\Omega)$. Then, for every $u\in W^{1,G}_0(\Omega)$, we have
$$
\langle \LL_{s_k} u, v_k\rangle \to   \langle \bar \LL_1 u, v\rangle.
$$
\end{lem}

\begin{proof}
First, observe that from Proposition \ref{bbm.norma}  it follows that $v\in W^{1,G}(\R^n)$ and $v=0$ a.e. in $\R^n\setminus \Omega$. Therefore, since $\Omega$ is Lipschitz, $v\in W^{1,G}_0(\Omega)$ and so everything is well defined.

Now, it is enough to show that
\begin{equation} \label{desig.1}
\langle \bar \LL_1 u, v\rangle  \leq    \liminf_{k\to\infty} \langle \LL_{s_k} u, v_k\rangle.
\end{equation}
In fact, if \eqref{desig.1} holds for every $u\in W^{1,G}_0(\Omega)$, then apply \eqref{desig.1} to $-u$ to get the reverse inequality.

Now, by Proposition \ref{bbm.norma} we have
$$
\HH_{1 ,\bar G} (u+\varepsilon v) \leq \liminf_{k\to\infty} \HH_{s_k}(u+\varepsilon v_k).
$$
The previous expression together with Remark \ref{rem.BBM.fijo} gives that
$$
\HH_{1,\bar G} (u+\varepsilon v) -\HH_{1,\bar G}(u) \leq \liminf_{k\to\infty} (\HH_{s_k}(u+\varepsilon v_k) - \HH_{s_k}(u)).
$$
Then, from Lemma \ref{asymptotic.development} we obtain
$$
\langle \bar \LL_1 u, v\rangle  + o(1) \leq \liminf_{k\to\infty} \langle \LL_{s_k} u, v_k\rangle + o(1),
$$
from where \eqref{desig.1} follows.
\end{proof}

%
%
%
%
%
%
%

We are now in position to prove Theorem \ref{main}.
\begin{proof}[Proof of Theorem \ref{main}]
Assume that $u_k\to u$ strongly in $L^G(\Omega)$. Then, since $\{u_k\}_{k\in\N}$ is uniformly bounded in $W^{s_k,G}_0(\Omega)$, by Proposition \ref{bbm.norma} we have that $u\in L^G(\Omega)$ and, up to a subsequence if necessary, we can assume that $u_k\to u$ a.e. in $\Omega$.

On the other hand, if we define $\eta_k:= \LL_{s_k} u_k\in W^{-s_k,\tilde G}(\Omega)\subset W^{-1,(p^-)'}(\Omega)$, then $\{\eta_k\}_{k\in\N}$ is bounded in $W^{-1,(p^-)'}(\Omega)$ and hence, up to a subsequence, there exists $\eta\in W^{-1,(p^-)'}(\Omega)$ such that $\eta_k \rightharpoonup \eta$ weakly in $W^{-1,(p^-)'}(\Omega)$.

Since $u_k$ solves \eqref{eq.s}, for any  $v\in C_c^\infty(\Omega)$
$$
0 = \langle \LL_{s_k} u_k ,v \rangle    -\int_\Omega f\left(x,\frac{u_k}{\|u_k\|_G}\right)v\,dx
$$
using the convergences, taking the limit $k\to\infty$ we get
$$
0=\langle \eta, v \rangle -\int_\Omega f\left(x,\frac{u}{\|u\|_G}\right)v\,dx.
$$
Now, we identify $\eta$: we will prove that 
\begin{equation} \label{will.prove}
\langle \eta, v \rangle = \langle \bar \LL_1 u,v\rangle \quad \text{for any }v\in C_c^\infty(\Omega).
\end{equation}
For that purpose we use the monotonicity of $\LL_{s_k}$ (see Proposition \ref{propo.mono}) and the fact that $u_k$ is solution of  \eqref{eq.s}. Indeed,
\begin{align*}
0&\leq \langle \LL_{s_k} u_k,u_k-v \rangle - \langle \LL_{s_k} v, u_k-v \rangle\\
&= 
\int_\Omega f\left(x,\frac{u_k}{\|u_k\|_G}\right) (u_k-v)\,dx  - \langle  \LL_{s_k} v, u_k - v\rangle.
\end{align*}
Hence taking the limit $k\to\infty$ and using Lemma \ref{key.lema} one finds that
\begin{align*}
0&\leq \int_\Omega f\left(x,\frac{u}{\|u\|_G}\right)(u-v) -   \langle  \bar \LL_1 v,u-v \rangle \\
&=\langle \eta, u-v \rangle -\langle \bar \LL_1 v,u-v \rangle.
\end{align*}
Consequently, if we take $v = u-tw$, $w\in W^{1,G}_0(\Omega)$ given and $t>0$, we obtain that
$$
0\leq \langle \eta, w \rangle - \langle \bar \LL_1 (u-tw), w\rangle
$$
taking $t\to 0^+$ gives that
$$
0\leq \langle \eta, w \rangle - \langle \bar \LL_1 u, w\rangle.
$$
From this it is easy to see that \eqref{will.prove} holds and the proof concludes.
\end{proof}

As mentioned in the beginning of the section, a consequence of the Theorem \ref{main} (and for technical reasons of Proposition \ref{bbm.norma}) is the stability as $s\uparrow 1$ for our eigenvalues problem:

\begin{cor} \label{corolario}
Let $\{\lam^s\}_{s\in(0,1)}$ be a family of eigenvalues of \eqref{eigen} and assume for some subsequence $\{s_k\}_{k\in\N}$ such that $s_k\uparrow 1$ that there exists $\lam^1$ with $\lam^{s_k}\to \lam^1$. Then $\lam^1$ is an eigenvalue of
\begin{equation}\label{eigen.1}
\begin{cases}
\bar \LL_1 u =   \lam^1\, g\left(\frac{|u|}{\|u\|_G} \right) \frac{u}{|u|} &\text{ in } \Omega\\
u=0 &\text{ on } \partial\Omega.
\end{cases}
\end{equation}
Moreover, if $\{u_{s_k}\}_k$ is the sequence of eigenfunctions of \eqref{eigen} corresponding to $\lam^{s_k}$ and $\|u_{s_k}\|_G=1$, then there exists $u\in W^{1,G}_0(\Omega)$ such that $u$ accumulation of the sequence $\{u_k\}_{k\in\N}$ in the $L^G(\Omega)-$topology and it is an eigenfunction with eigenvalue $\lam^1$. 
\end{cor}

\subsection{Stability of variational eigenvalues}
The purpose of this section is to investigate further the result obtained in Corollary \ref{corolario} and specialize the case where the eigenvalue sequence $\lambda^s$ in Corollary \ref{corolario} is given by the $k^{th}$ variational eigenvalue of \eqref{eigen}. We show in this subsection that if one considers $\{\lambda_k^s\}_{s>0}$ and let $s\uparrow 1$, then the limit eigenvalue (which is an eigenvalue of \eqref{eigen.1} by Corollary \ref{corolario}) is in fact the $k^{th}$ variational eigenvalue of \eqref{eigen.1}.

The result in this subsection are inspired by \cite{FBSS} and the main ideas can be traced back to \cite{Champion-DePascale}.

Let us define, for $0<s\le 1$, 
\[
\mathcal{A}^k_{s,G}=\left\{A\subset W^{s,G}_0(\Omega):
\begin{array}{c}
A=-A,\text{ is closed and bounded in } W_0^{s,G}(\Omega)\\
\|u\|_G=1,\:\forall u\in A\text{ and }\gamma(A)\geq k
\end{array} 
\right\}.
\]
Here $\gamma(A)$ stands for the \emph{genus} of $A$, see \cite{R}.

Next, we define the functional $J_k^s:\mathcal{K}_{\text{sym}}(\Omega)\rightarrow[0,\infty]$, $\mathcal{K}_{\text{sym}}(\Omega)$ being the symmetric compact subsets of $L^{G}(\Omega)$, given by
$$
J_k^s(A)= \begin{cases}
\sup_{v\in A}\HH_s(v) &\text{if } A\in \mathcal{A}_{s,G}^k(\Omega)\\
						         \infty & \text{otherwise.}
\end{cases} 
$$

With these notations, observe that the $k^{th}$ variational eigenvalue of \eqref{eigen} is then given by 
\begin{equation}\label{eq.eigen}
\lambda^s_k=\inf_{A\in\mathcal{K}_{\text{sym}}(\Omega)}J_k^s(A).
\end{equation}

\begin{thm}\label{thm.abst}
For any $k\in\N$, the sequence $\{J_k^s\}_{s>0}$ is equicoercive, $\Gamma-\liminf_{s\uparrow 1} J_k^s\geq \bar J_k$ and
$$
\lim_{s\uparrow 1}\left(\inf_{A\in\mathcal{K}_{\text{sym}}(\Omega)}J_k^s(A)\right) = \inf_{A\in\mathcal{K}_{\text{sym}}(\Omega)}\bar J_k(A),
$$
where 
$$
\bar J_k(A) =\begin{cases}
\sup_{v\in A}\HH_{1,\bar G}(v) &\text{if } A\in \mathcal{A}_{1,G}^k(\Omega)\\
						         \infty & \text{otherwise}
\end{cases} 
$$
and $\bar G$ is the Young function given in \eqref{bG}.
\end{thm}

\begin{proof}
The proof is divided into three steps.

\medskip

{\em Step1:} equicoercivity. 
\smallskip

Let $A\subset \{J_k^s\le\mu\}$, then $A\in \A_{s, G}^k$ and
\begin{equation}\label{Fnmu}
\HH_s(v)\leq\mu \quad \text{for all } v\in A. 
\end{equation}

Take now $s_0\in (0,1)$ fixed, we can assume that $s_0<s<1$. Then, by \cite[Theorem 5.2 and Theorem 6.1]{FBS}, we have that
\begin{equation}\label{bbm1}
\HH_{s_0}(v)\le C \HH_s(v)
\end{equation}
with $C$ independent of $s$.

From \eqref{Fnmu} and \eqref{bbm1} we have $\HH_{s_0}(v)\leq C$ for every $v\in A$ where $C$ depends only on $\mu, s_0, G$ and $\Omega$. We define
$$
K := \{v\in L^{G}(\Omega): \HH_{s_0}(v)\leq C\}.
$$
By the Sobolev embedding (see \cite[Theorem 3.1]{FBS}), we obtain that $K$ is compact in $L^{G}(\Omega)$, so $\{J_k^s \leq\mu\}\subset \{A\in \Ks(\Omega)\colon A\subset K\}$ which is a compact subset of $\Ks(\Omega)$, so that the family $J_k^s$ is equicoercive.

\medskip

{\em Step 2:}  $\Gamma-\liminf$.
\smallskip

 Let $A_0\in \Ks(\Omega)$ and $\{A_s\}_{s>0}\subset \Ks(\Omega)$ be a sequence such that $A_s\to A_0$ ($s\uparrow 1$)  in Hausdorff distance. We shall prove that
$$\liminf_{s\uparrow 1} J_k^s(A_s)\geq \bar J_k(A_0).
$$
 
Without loss of generality, we may assume that there exists a constant $C>0$ such that
 $J_k^s(A_s)<C$ for every $s\in (0,1)$. Observe that this implies that $A_s \in\A^k_{s,G}$. 
 
We will show $\gamma(A_0)\geq k$. To this end, take an open neighborhood $N$ of $A_0$ in $L^{G}(\Omega)$ such that $\gamma(A_0)=\gamma(\overline{N})$. Since $A_s\to A_0$ in Hausdorff distance, $A_s\subset N$ for any $s$ sufficiently close to 1. Therefore, by the monotonicity of the genus we get
 $$
 k\leq\gamma(A_s)\leq\gamma(\overline{N})=\gamma(A_0).
$$
Now, for any $u\in A_0$ there exists a sequence $u_s\in A_s$ such that $u_s\to u$ in $L^G(\Omega)$. By Proposition \ref{bbm.norma} we have that
$$
\HH_{1,\bar G}(u) \leq\liminf_{s\uparrow 1} [u_s]_{s, G}\leq \liminf_{s\uparrow 1}\sup_{v\in A_s}[v]_{s, G}= \liminf_{s\uparrow 1} J_k^s(A_s)
$$
for all $u\in A_0$. Taking supremum we obtain the desired result.

\medskip

{\em Step 3:} In light of the previous two steps, it is easy to see that it only remains to prove that
$$
\limsup_{s\uparrow 1} \left(\inf_{A\in \Ks(\Omega)} J_k^s(A)\right) \leq \inf_{A\in \Ks(\Omega)}\bar J_k(A).
$$

We fix $\delta>0$ small and let $A_0\in \Ks(\Omega)$ be such that
$$
\inf_{A\in \Ks(\Omega)}\bar J_k(A)\geq \bar J_k(A_0)-\delta
$$
Since $A_0$ is compact in $L^G(\Omega)$, there exist $u^1,u^2,\dots,u^m\in A_0$ such that 
$$
A_0\subset \bigcup_{i=1}^{m} B_{L^G(\Omega)}(u^i, \delta).
$$ 
Recall that $W^{1,\bar G}_0(\Omega) = W^{1,G}_0(\Omega)\subset W^{s, G}_0(\Omega)$ and that $\|u^i\|_{s, G}\to \|\nabla u^i\|_{\bar G}$ as $s\uparrow 1$ for every $i=1,\dots,m$.

Next, for every $n\in \N$, we define
$$
C= \overline{\text{Co}(\{\pm u^{i}; i=1,\dots,m\})},
$$
where $\text{Co}(A)$ is the convex hull of the set $A$.

Note that $C\subset L^G(\Omega)$ is compact.

 We denote by $\Pi$ the projection onto $C$, for the norm of $L^G(\Omega)$.  That is $\Pi\colon L^G(\Omega)\to L^G(\Omega)$ such that $\Pi v\in C$ and
$$
\|\Pi v - v\|_G\le  \inf_{u\in C} \|u-v\|_G.
$$
Observe that $\Pi$ is well defined since $C$ is compact and $\|\cdot\|_G$ is uniformly convex. Moreover, $\Pi$ is Lipschitz with Lipschitz constant 1.

Now, we want to prove that $\Pi(A_0)$ is far from $0$. To this end, if $v\in A_0$, we have that there exists $i\in \{1,\dots,m\}$ such that $\|v-u^i\|_G< \delta$. Therefore
$$
\|\Pi v\|_G \ge \|u^{i}\|_G - \|\Pi u^{i} - u^{i}\|_G - \|\Pi v - \Pi u^{i}\|_G\ge 1-\delta.
$$
So, $\Pi(A_0)\subset C-B(0,1-\delta)$ and, since $C\in \Ks(\Omega)$, it follows that $\Pi(A_0)\in \Ks(\Omega)$ and $\gamma(\Pi(A_0))\geq k$, therefore if we define
$$
\tilde A_0:= \left\{\frac{v}{\|v\|_G}\colon v\in \Pi(A_0)\right\},
$$
we get that $\tilde A_0\in \A^k_{s, G}(\Omega)$.

Now, for every $v\in \Pi(A_0)\subset C$, we obtain
\begin{align*}
\HH_s\left(\frac{v}{\|v\|_G}\right)&=\frac{1}{\|v\|_G}\HH_s(v) \leq\frac{1}{1-\delta} \sup_{v\in C}\HH_s(v)\le 
\frac{1}{1-\delta}\max_{1\leq i\leq m}\HH_s(u^i).
\end{align*}
So,
$$
J_s^{k}(\tilde A_0)\leq\frac{1}{1-\delta}\max_{1\leq i\leq m}\HH_s(u^{i}).
$$ 
As a consequence,
\begin{align*}
\limsup_{s\uparrow 1}(\inf J_s^{k})&\leq \limsup_{s\uparrow 1} J_s^{k}(\tilde A_0)\\
&\leq\frac{1}{1-\delta}\max_{1\leq i\leq m}\HH_{1,\bar G}(u^{i})\\
&\leq\frac{1}{1-\delta}\sup_{A_0}\HH_{1,\bar G}\\
&\leq\frac{1}{1-\delta}(\inf \bar J^{k}+\delta).
\end{align*}
The conclusion of step 3  follows by letting $\delta\to 0$.
\end{proof}

\begin{cor}\label{cor.eigen}
An immediate corollary of Theorem \ref{thm.abst} is the fact that for any fixed $k\in\N$ the $k^{th}$ variational eigenvalue $\lambda_k^s$ of \eqref{eigen} converges as $s\uparrow 1$ the the $k^{th}$ variational eigenvalue of the limit equation  \eqref{eigen1}.
\end{cor}

\subsection*{Acknowledgements.} This work was partially supported by CONICET under grant  PIP No. 11220150100032CO, by ANPCyT under grants PICT 2016-1022 and PICT 2019-3530 and by the University of Buenos Aires under grant 20020170100445BA. 

All three authors are members of CONICET.


\begin{thebibliography}{00}

\bibitem{Ap04} Applebaum, D., \textit{L\'evy processes --- From Probability to Finance and Quantum Groups}, Notices AMS 51 (2004), 1336-1347.

\bibitem{ACPS} Alberico, A., Cianchi, A., Pick, L.,  Slavikov\'a, L. (2021). \emph{On fractional Orlicz-Sobolev spaces}. Analysis and Mathematical Physics, 11(2), 1-21.

\bibitem{ACPS1}
Alberico, A., Cianchi, A., Pick, L.,  Slavikov\'a, L.,  \emph{On the limit as $ s\to 1^-$ of possibly non-separable fractional Orlicz-Sobolev spaces}. Rendiconti Lincei-Matematica e Applicazioni, 31(4), 879-899. (2021).

\bibitem{BS}
Bahrouni, S., Salort, A.,  \emph{Neumann and Robin type boundary conditions in Fractional Orlicz-Sobolev spaces}. ESAIM: Control, Optimisation and Calculus of Variations, 27, S15.

\bibitem{BBM} Bourgain, J., Brezis, H., Mironescu, P., \emph{Another look at Sobolev spaces}. (2001): 439-455.

\bibitem{BBM1}
Bourgain, J., Brezis, H.,  Mironescu, P., \emph{Limiting embedding theorems for $W^{s,p}$ when $s\uparrow 1$ and applications}. Journal d'Analyse Math\'ematique, 87(1), 77-101. (2002).


\bibitem{BV}
Bucur, C., Valdinoci, E., \emph{Nonlocal diffusion and applications.} Vol. 1. Switzerland: Springer International Publishing, 2016.

\bibitem{brezis} Brezis, H.,  \emph{Functional analysis, Sobolev spaces and partial differential equations}. Springer Science \& Business Media. (2010).

\bibitem{CSS} Cantizano, N., Salort, A. ,   Spedaletti, J. \emph{Continuity of solutions for the $G$-Laplacian operator}. Proceedings of the Royal Society of Edinburgh Section A: Mathematics, 1-28.

\bibitem{Champion-DePascale} Thierry Champion and Luigi De Pascale, \emph{Asymptotic behaviour of nonlinear eigenvalue problems involving $p-$laplacian-type operators}, Proceedings of the Royal Society of Edinburgh 137A (2007), 1179--1195.

\bibitem{CT16} Cont, R., Tankov, P., \emph{Financial Modelling With Jump Processes}, Financial Mathematics Series. Chapman \& Hall/CRC, Boca Raton, FL, 2004.





\bibitem{FBS} Fern\'andez Bonder, J., Salort, A., \emph{Fractional order Orlicz-Sobolev spaces}. Journal of Functional Analysis 277.2 (2019): 333-367.


\bibitem{FBPS}
Fern\'andez Bonder, J., Pinasco, J. P., Salort, A., \emph{Quasilinear eigenvalues}. Rev. Un. Mat. Argentina 56 (2015), no. 1, 1--25

\bibitem{FBPLS} 
Fern\'andez Bonder, J., P\'erez-Llanos, M., Salort, A., \emph{A H\"older infinity Laplacian obtained as limit of Orlicz fractional Laplacians} Revista Matem\'atica Complutense DOI: 10.1007/s13163-021-00390-2


\bibitem{FBS1}
Fern\'andez Bonder, J., Salort, A., \emph{Magnetic fractional order Orlicz-Sobolev spaces.} Studia Math. 259 (2021), no. 1, 1--24.


\bibitem{FBSV2} Fern\'andez Bonder, J., Salort, A., Vivas, H., \emph{Interior and up to the boundary regularity for the fractional $g$-Laplacian: the convex case}. arXiv preprint arXiv:2008.05543 (2020).

\bibitem{FBSV3} Fern\'andez Bonder, J., Salort, A., Vivas, H., \emph{Global H\"older reglarity for eigenfunctions of the fractional $g$-Laplacian}. arXiv preprint arXiv:2112.00830 (2020).

\bibitem{FBSS} Fern\'andez Bonder, J., Silva, A. and Spedaletti, J. \emph{Gamma convergence and asymptotic behavior for eigenvalues of nonlocal problems}. Discrete and Continuous Dynamical Systems,  {\bf 41} (2021), no. 5, pp. 2125--2140.



\bibitem {FrLi} Franzina, G., Lindqvist, P., \emph{An eigenvalue problem with variable exponents}. Nonlinear Analysis: Theory, Methods $\&$ Applications 85 (2013): 1-16.

\bibitem{GHLMS}
Garc\'ia-Huidobro, M., Le, V. K., Man\'asevich, R., Schmitt, K., \emph{On principal eigenvalues for quasilinear elliptic differential operators: an Orlicz-Sobolev space setting}. Nonlinear Differential Equations and Applications NoDEA, 6(2), 207-225. (1999).

\bibitem{Ga}
Garofalo, N., \emph{Fractional thoughts.} arXiv preprint arXiv:1712.03347 (2017).


\bibitem{GKS}
Giacomoni, J., Kumar, D., Sreenadh, K., \emph{Boundary regularity results for strongly nonhomogeneous $(p, q)-$fractional problems}. arXiv e-prints (2021): arXiv-2102.

\bibitem{GM}
Gossez, J.P., Mansevich, R. \emph{On a nonlinear eigenvalue problem in OrliczSobolev spaces.} Proceedings. Section A, Mathematics-The Royal Society of Edinburgh, 132(4), 891. (2002).


\bibitem {KR} Krasnoselski\u{\i}, M., Ruticki\u{\i}, J., \emph{Convex functions and Orlicz spaces.} P. Noordhoff Ltd., Groningen, 1961.

\bibitem{Lamperti} Lamperti, J., \emph{On the isometries of certain function-spaces}, Pacific J. Math, 8(3) (1958), 459-466. 7

\bibitem{L}
Lindqvist, P., \emph{A nonlinear eigenvalue problem}. Topics in mathematical analysis, 3, 175-203. (2008).

\bibitem{L2}
Lindqvist, P.  \emph{Notes on the p-Laplace equation} (No. 161). University of Jyv\"askyl\"a. (2017).

\bibitem{MMP} Motreanu, D., Motreanu, V., Papageorgiou N., \emph{Topological and Variational Methods with Applications to Nonlinear Boundary Value Problems}, Springer-Verlag New York (2014)

\bibitem{MSV} Molina, S., Salort, A., Vivas, H., \emph{Maximum principles, Liouville theorem and symmetry results for the fractional  $g$-Laplacian}. Nonlinear Anal. 212 (2021), 112465. 

\bibitem{P}
Peral, I., \emph{Multiplicity of solutions for the $p-$Laplacian}. International Center for Theoretical Physics Lecture Notes, Trieste. (1997).

\bibitem{Ponce}
Ponce, A.,  \emph{A new approach to Sobolev spaces and connections to $\Gamma-$convergence.} Calc. Var. Partial Differential Equations, 19(3), 229-255. (2004).

\bibitem {R} Rabinowitz, P.,  \emph{Minimax methods in critical point theory with applications to differential equations}, CBMS Regional Conference Series in Mathematics, vol. 65, Published for the Conference Board of the Mathematical Sciences, Washington, DC; by the American Mathematical Society, Providence, RI, 1986. MR 845785  

\bibitem{S}
Salort, A., \emph{Eigenvalues and minimizers for a non-standard growth non-local operator}. Journal of Differential Equations, 268(9), 5413-5439. (2020).


\bibitem{SV}
Salort, A., Vivas, H., \emph{Fractional eigenvalues in Orlicz spaces with no $\Delta_2$ condition}. Journal of Differential Equations, 327, 166-188. (2022).


\bibitem{ST94} Samorodnitsky, G., Taqqu, M., \emph{Stable Non-Gaussian Random Processes: Stochastic Models With Infinite Variance}, Chapman and Hall, New York, 1994.

\end{thebibliography}
\end{document}